\documentclass{amsart}
\usepackage{amssymb, amscd, ascmac}
\usepackage{pb-xy,pb-diagram}
\newtheorem{theorem}[subsection]{Theorem}
\newtheorem{proposition}[subsection]{Proposition}
\newtheorem{lemma}[subsection]{Lemma}
\newtheorem{corollary}[subsection]{Corollary}
\theoremstyle{definition}

\newtheorem{remark}[subsection]{Remark}
\newcommand{\PY}{Pirutka and Yagita }
\newcommand{\CS}{Colliot-Th\'el\`ene\ and Szamuely }

\title[The cycle map]{Representation theory and the cycle map of a classifying space}
\author{Masaki Kameko}
\address{
Department of Mathematical Sciences,
Shibaura Institute of Technology,
307 Minuma-ku Fukasaku, Saitama-City 337-8570, Japan}
\email{kameko@shibaura-it.ac.jp}
\thanks{The author is partially supported 
by the Japan Society for the Promotion of Science, 
Grant-in-Aid for Scientific Research (C) 25400097.}
\begin{document} 
\begin{abstract}
We compute the Chern subgroup of the $4$-th integral cohomology group of  a certain classifying space and show that it is a proper subgroup. Such a classifying space gives us  new counterexamples for the integral Hodge and Tate conjectures modulo torsion.
\end{abstract}
\maketitle
\section{Introduction}
\label{section:1}
Let $G$ be a reductive complex linear algebraic group.
In  
\cite{totaro-1997},
Totaro defined the Chow group $CH^iBG$ of its classifying space $BG$ 
and studied the cycle map to the integral cohomology $H^{2i}(BG;\mathbb{Z})$. 
Approximating the classifying space with smooth projective varieties, 
we obtain interesting examples.
For instance,  counterexamples for the integral Hodge conjecture 
due to Atiyah and Hirzebruch
\cite{atiyah-hirzebruch-1962}
are obtained by studying the cycle map
\[
cl:CH^iBG\to H^{2i}(BG;\mathbb{Z})
\]
for $i=2$, $G=(\mathbb{Z}/p)^3 \times \mathbb{G}_m$ where $p$ is 
an odd prime number, 
$(\mathbb{Z}/p)^3$ is the elementary abelian $p$-group of rank $3$ and 
$\mathbb{G}_m$ is the multiplicative group of non-zero complex numbers.
Changing the base field to the algebraic closure of a finite field whose characteristic 
is prime to $p$, \CS obtained counterexamples 
for the  integral Tate conjecture in \cite{colliot-thelene-szamuely-2010}.
However, certain classes of varieties are known for which the integral Tate conjecture holds
and it is also an important subject of study, see, for example, Charles and Pirutka, \cite{charles-pirutka-2015}.
We refer the reader to  the introduction of Antieau's paper \cite{antieau-2015}
 for a detailed account of integral Hodge and Tate conjectures.
Since a reductive complex linear algebraic group is 
homotopy equivalent to its maximal compact subgroup,
as far as its cohomology is concerned, we may consider the corresponding 
compact Lie group instead of the reductive complex linear algebraic group.
In this paper, we  consider the special unitary group $SU(p^2)$ instead of 
the special linear group $SL_{p^2}$
and by abuse of notation, we write $CH^iBSU(p^2)$ for 
$CH^iBSL_{p^2}$ and so on.

Recently, using exceptional groups $G_2, F_4, E_8$ for $p=2, 3, 5$, respectively, \PY 
constructed counterexamples for the integral Hodge and Tate conjectures modulo torsion 
in \cite{pirutka-yagita-2014}.
In \cite{kameko-2015},
the author generalized the result of \PY using compact Lie groups
$SU(p)\times SU(p)/\mu_p$ to all prime numbers $p$, 
where $\mu_p$ is the cyclic group of order $p$ 
which acts diagonally on $SU(p)\times SU(p)$.
Very recently, in \cite{antieau-2015}, 
Antieau gave rather different and very interesting counterexamples 
for the integral Hodge and Tate conjectures modulo torsion for $p=2, 3$ 
by studying Chern classes of  $SU(8)/\mu_2$ and $SU(9)/\mu_3$, 
where $\mu_p$ is the subgroup of the center of $SU(p^2)$ of order $p$.
By the Riemann--Roch formula without denominators, 
Chow group $CH^2BG$ is generated by Chern classes of complex representations. 
So, the question on the surjectivity of the cycle map 
\[
cl:CH^2BG \to H^4(BG,\mathbb{Z})
\]
is equivalent to the question on the integral cohomology group $H^4(BG;\mathbb{Z})$
being generated by Chern classes of complex representations.
We refer the reader to Totaro' paper \cite{totaro-1999} and his book  \cite{totaro-2014} 
for details of these facts.
In \cite{antieau-2015}, Antieau calculated second Chern classes for all irreducible complex representations 
of $SU(8)/\mu_2$ and $SU(9)/\mu_3$
using a computer algebra system and showed that these second 
Chern classes do not generate the $4$-th integral cohomology groups.
Also, Antieau conjectured similar results hold for all odd prime numbers, that is, 
Theorem~\ref{theorem:1.1}
 \cite[Conjecture 2.10]{antieau-2015} below.
 
In this paper, 
we call the subgroup of the integral cohomology group $H^i(BG;\mathbb{Z})$
generated by all complex representations
the Chern subgroup of $H^i(BG;\mathbb{Z})$.
The purpose of this paper is to prove Antieau's conjecture. But, 
we do not only prove Antieau's conjecture  but also give a simpler proof
for the case $p=3$ without using a computer algebra system.
\begin{theorem}\label{theorem:1.1}
Let $p$ be an odd prime number.
Let $\mu_p$ be the subgroup of the center of $SU(p^2)$ of order $p$.
Then, the Chern subgroup of 
\[
H^4(BSU(p^2)/\mu_p;\mathbb{Z})=\mathbb{Z}
\]
is
\[
p\cdot H^4(BSU(p^2)/\mu_p;\mathbb{Z})=p\cdot \mathbb{Z}.
\]
\end{theorem}

As a corollary, we have the following result on the cycle map.
\begin{corollary}
\label{corollary:1.2}
Let $p$ be an odd prime number.
Let $\mu_p$ be the subgroup of the center of $SU(p^2)$ of order $p$.
Then, the cycle map 
\[
cl:CH^2BSU(p^2)/\mu_p \to H^4(BSU(p^2)/\mu_p;\mathbb{Z})=\mathbb{Z}
\]
is not surjective.
\end{corollary}

One of the key ingredients in Antieau's paper is the fact,
(\cite[Proposition 2.7]{antieau-2015}),
that the induced homomorphism
\[
\pi^*:H^4(BSU(p^2)/\mu_p;\mathbb{Z}) \to H^{4}(BSU(p^2);\mathbb{Z})
\]
induced by the obvious projection $\pi:SU(p^2)\to SU(p^2)/\mu_p$
is an isomorphism for all odd prime numbers $p$,
where, by abuse of notation, we denote  the map between classifying spaces 
$BSU(p^2)\to BSU(p^2)/\mu_p$ 
induced by the projection 
$\pi:SU(p^2)\to SU(p^2)/\mu_p$ by the same symbol $\pi$.
Antieau used $SU(p^2)$ to detect the integral cohomology group
$H^4(BSU(p^2)/\mu_p;\mathbb{Z})$.
In this paper, 
we replace $SU(p^2)$ by 
\[
S^1=\{ x \in \mathbb{C}\;|\; |x|=1\}.
\]
The complex representation ring $R(S^1)$ is $\mathbb{Z}[z, z^{-1}]$ 
and the integral cohomology ring 
$H^{*}(BS^1;\mathbb{Z})$ is $\mathbb{Z}[t]$ where $t$ is 
a generator of $H^2(BS^1;\mathbb{Z})=\mathbb{Z}$ and 
$z$ is represented by the canonical line bundle over the infinite dimensional complex 
projective space $\mathbb{C}P^\infty=BS^1$.
Both the complex representation ring of  $S^1$
and the integral cohomology ring of $BS^1$ are much simpler than 
those of $SU(p^2)$, $BSU(p^2)$, yet there exists a map
\[
\varphi_1:S^1 \to SU(p^2)
\]
such that the induced homomorphism
$(\pi\circ \varphi_1)^*$ detects the $4$-th integral cohomology group
of $BSU(p^2)/\mu_p$.
Because of the simplicity of 
the complex representation ring of $S^1$ and the integral cohomology ring of $BS^1$,
we can compute complex representations and their Chern classes 
without using a computer algebra system.
Also, in this paper, we consider algebra generators of the complex representation ring 
$R(SU(p^2))$ rather than 
irreducible representations. It also reduces the amount of our computations.

This paper is organized as follows:
In Section~\ref{section:2}, we give a necessary condition for an element $x$ in the complex representation ring
$R(SU(p^2))$ of $SU(p^2)$ to be in the image of
the induced homomorphism 
$$
\pi^*:R(SU(p^2)/\mu_p) \to R(SU(p^2)).
$$
In  Section~\ref{section:3}, 
we define the map $\varphi_1:S^1 \to SU(p^2)$ above. 
Then, we compute the second Chern class of $\varphi_1^*(x) \in R(S^1)$ modulo $p$ 
and complete the proof of Theorem~\ref{theorem:1.1}.

The author was informed by
Benjamin Antieau that  Arnav Tripathy also proved Theorem~\ref{theorem:1.1} independently. (\cite{tripathy-2016})


The author would like to thank Benjamin Antieau for sending an early draft of his preprint, giving comments and suggestions for the draft of this paper
and informing the author of Tripathy's work.  
Also the author would like to thank 
 the anonymous referee for his/her 
report, especially informing the author on the recent work of 
Charles and Pirutka, and comments
which improved the presentation of this paper.

\section{Representations}
\label{section:2}
In this section, we consider a necessary condition for an element $x$ in the complex representation ring $R(SU(p^2))$
to be in the image of the induced homomorphism
\[
\pi^*:R(SU(p^2)/\mu_p)\to R(SU(p^2)).
\]
We use the following commutative diagram:
\[
\begin{diagram}
\node{\mu_p} \arrow{e,t}{\Delta_0} \arrow{s,l}{\Delta_1} \node{T} \arrow{s,r}{\iota_0}
\\
\node{SU(p^2)} \arrow{e,t}{\iota_1} \arrow{s,l}{\pi} \node{U(p^2)} \\
\node{SU(p^2)/\mu_p,}
\end{diagram}
\]
where $T=S^1\times \cdots \times S^1$ (the product of $p^2$ copies of $S^1$'s)
is the maximal torus of the unitary group $U(p^2)$ 
consisting of diagonal matrices, $\mu_p$ is the cyclic group of order $p$, consisting of $p$-th roots of unity in $S^1\subset \mathbb{C}$, $\iota_0, \iota_1$ are obvious inclusion maps, and $\Delta_0$ is the diagonal map 
\[
\Delta_0(x)=(x,x, \dots,  x).
\]
Complex representation rings of $\mu_p, T, SU(p^2), U(p^2)$ are well-known. 
See for instance, 
Br\"{o}cker and tom Dieck's book 
 \cite[Chapter II, Section 8]{brocker-tomdieck-1995} and
Husemoller's book \cite[Chapter 14]{husemoller-1994}.
We recall that 
\begin{align*}
R(\mu_p)&
=\mathbb{Z}[\zeta]/(\zeta^p-1)
=\mathbb{Z}\{ 1,\zeta, \dots, \zeta^{p-1}\},
 \\
R(T)&
=\mathbb{Z}[ z_1, z_1^{-1}, \dots, z_{p^2}, z_{p^2}^{-1}]
=\mathbb{Z}[z_1, \dots, z_{p^2-1}, z_{p^2}, \lambda_{p^2}^{-1}], 
\\
R(U(p^2))&
=\mathbb{Z}[\lambda_1, \dots, \lambda_{p^2-1}, \lambda_{p^2}, \lambda_{p^2}^{-1}], 
\\
R(SU(p^2))&=\mathbb{Z}[\lambda_1, \dots, \lambda_{p^2-1}], 
\end{align*}
where
$\zeta$ is represented by the inclusion of $\mu_p$ to $S^1$,
 $z_i$ is $p_i^*(z)$, 
 $p_i:T\to S^1$ is the projection onto the $i$-th factor for $i=1, \dots, p^2$, 
$\lambda_{\ell}$ is the $\ell$-th symmetric function of $z_1, \dots, z_{p^2}$ and 
$\lambda_\ell$ is represented by the $\ell$-th exterior power representation.
It is clear from the definition of $\Delta_0$ that $\Delta_0^{*}(z_i)=\zeta$ for $i=1, \dots, 
p^2$. Hence, we have
\[
\Delta_1^*(\lambda_\ell)=\binom{p^2}{\ell} \zeta^\ell
\]
and
\[
\Delta_1^*(\lambda_{\ell_1}\cdots \lambda_{\ell_r})
=\binom{p^2}{\ell_1} \cdots \binom{p^2}{\ell_r}\zeta^{\ell_1+\cdots +\ell_r}.
\]

If $x\in R(SU(p^2))$ is in the image of the induced homomorphism 
\[
\pi^*:R(SU(p^2)/\mu_p) \to R(SU(p^2)),
\]
then $\Delta_1^*(x)$ is in 
\[
\mathbb{Z}\{ 1\} \subset \mathbb{Z}\{ 1, \zeta, \dots, \zeta^{p-1}\}=R(\mu_{p}).
\]


\begin{remark}
\label{remark:2.1}
Since 
\[
\displaystyle \Delta_1^{*}(\lambda_p)=\binom{p^2}{p} \zeta^{p}  =\binom{p^2}{p} 
\cdot 1
\in \mathbb{Z}\{1\},
\]
and since
\[
\displaystyle \Delta_1^*(\lambda_1^p)=\binom{p^2}{1}^p \zeta^p=\binom{p^2}{1}^p \cdot 1
\in \mathbb{Z}\{1\},
\]
there exist complex representations $y_1$, $y_2$ of $SU(p^2)/\mu_p$ such that
$\pi^*(y_1)=\lambda_p$ and $\pi^*(y_2)=\lambda_1^p$.
\end{remark}

Now, we consider a basis for  ${R}(SU(p^2))$.
Let $L=(\ell_1, \dots, \ell_r)$ be a  non-decreasing finite sequence of integers in $\{ 1, \dots, p^2-1\}$.
We define a monomial $\lambda_L$ to be 
\begin{align*}
\lambda_{\emptyset}&=1, \\
\lambda_{(\ell)}&=\lambda_\ell, \\
\lambda_{(\ell_1, \dots, \ell_r)}&=\lambda_{\ell_1}\cdots \lambda_{\ell_r} \quad \mbox{for $r\geq 2$.}
\end{align*}
Then, any element $x$ in ${R}(SU(p^2))$ is a linear combination of $\lambda_{L}$'s
and we have 
\[
x=\sum_{L} \alpha_{L}(x) \lambda_L, \quad \alpha_L(x) \in \mathbb{Z},
\]
where $L$ ranges over the set of non-decreasing finite sequences of positive
integers less than $p^2$
and $\alpha_{L}(x)$'s are zero except for a finite number of $L$'s.
 In this section, we prove the following proposition, which
gives us a necessary condition for $x\in {R}(SU(p^2))$ to be 
 in the image of $\pi^*$. The use of this proposition in the proof of Theorem~\ref{theorem:1.1}
 was suggested by Antieau to the author. Also, the author learned its proof from Antieau.
 

\begin{proposition}
\label{proposition:2.2}
Let $L=(\ell_1, \dots, \ell_r)$. 
Suppose that $x \in \mathrm{Im}\, \pi^*$.
If
$\ell_1+\cdots+\ell_r\not \equiv 0 \mod p$, then
\[
\alpha_L(x)=0.
\]
\end{proposition}
 
\begin{proof}
Let $V'$ be a complex representation of $SU(p^2)/\mu_p$. Then $\pi^*(V')$ is a direct sum of 
irreducible representations, say $\pi^*(V')=V_1\oplus \cdots \oplus V_s$. Since $\mu_p$ acts on 
$\pi^*(V)$ trivially, it acts trivially on each $V_i$. 
Therefore, there exists $V_i'$ in $R(SU(p^2)/\mu_p)$ such that 
$\pi^*(V_i')=V_i$. 
Since a complex representation ring is additively generated by irreducible representations,
to prove the proposition, it suffices to prove it for irreducible representations.
Let $V$ be an irreducible representation of $SU(p^2)$.
It is a direct summand of an $n$-fold tensor product of the standard representation $\lambda_1$
for some $n$.
The irreducible representation $V$ corresponds to a sum of $\dim(V)$ 
homogeneous monomials of degree $n$ in $R(T')$, where $T'=T \cap SU(p^2)$ is the maximal torus of $SU(p^2)$, so that
it is a linear combination of $\lambda_{\ell_{1}}\cdots \lambda_{\ell_r}$'s such that  $\ell_1+\cdots+\ell_r
= n $, where $\ell_i \in \{ 1, \dots, p^2\}$.
Since 
\[
\Delta_1^*(V)=(\dim V) \zeta^n,
\]
it is in 
\[
\mathbb{Z}\{1\} \subset \mathbb{Z}\{ 1, \zeta, \dots, \zeta^{p-1}\} =R(\mu_p)
\]
if and only if $n$ is divisible by $p$. This completes the proof.
\end{proof}

\section{Chern classes}
\label{section:3}
In this section, we prove Theorem~\ref{theorem:1.1}.
We define a map $\varphi_0:S^1\to T$ by
\[
\varphi_0(x)=(x, x^{-1}, 1, \dots, 1).
\]
The map $\varphi_1:S^1\to SU(p^2)$ mentioned 
in  Section~\ref{section:1} is defined by the following 
commutative diagram:
\[
\begin{diagram}
\node{S^1} \arrow{e,t}{\varphi_0} \arrow{s,l}{\varphi_1} \node{T} \arrow{s,r}{\iota_0} \\
\node{SU(p^2)} \arrow{e,t}{\iota_1} \node{U(p^2).}
\end{diagram}
\]
As in Section~\ref{section:2}, 
we denote algebra generators of $R(T)$ by $z_1, \dots, z_{p^2}$, 
those of $R(SU(p^2))$ by $\lambda_1, \dots, \lambda_{p^2-1}$ and 
those of $R(U(p^2))$ by $\lambda_1, \dots, \lambda_{p^2}, \lambda_{p^2}^{-1}$, where 
$\iota_0^*(\lambda_i)$ is the $i$-th symmetric polynomial in $z_1, \dots, z_{p^2}$
and $\iota_1^*(\lambda_\ell)=\lambda_\ell$ ($\ell=1, \dots, p^{2}-1$) and 
$\iota_1^*(\lambda_{p^2})=1$.
Let $z$ be the generator of $R(S^1)=\mathbb{Z}[z,z^{-1}]$ represented by the canonical line bundle and 
let $t$ be the generator of $H^{*}(BS^1;\mathbb{Z})$ as we did in Section~\ref{section:1}.
It is clear that $\varphi_0^*(z_1)=z$, $\varphi_0^*(z_2)=z^{-1}$ and 
$\varphi_0^*(z_i)=1$ for $i=3, \dots, p^2$.
For a vector bundle $\xi$ over a space $X$, we denote by 
\[
c(\xi)=1+c_1(\xi)+c_2(\xi)+\cdots \in H^{*}(X;\mathbb{Z})
\]
the total Chern class of $\xi$, where $c_i(\xi)$ denotes the $i$-th Chern class in 
$H^{2i}(X;\mathbb{Z})$.

\begin{proposition}
\label{proposition:3.1}
For $\ell=1, \dots, p^2-1$, we have
$
c_1(\varphi_1^{*}(\lambda_\ell))=0
$
 and 
$
c_2(\varphi_1^{*}(\lambda_{pk}))\equiv 0 \mod p,
$
where $0< k < p$.
\end{proposition}

\begin{proof}
The fact $c_1(\varphi_1^{*}(\lambda_\ell))=0$ is clear from the fact that $H^2(BSU(p^2);\mathbb{Z})=\{0\}$.
Let us use the convention that $\displaystyle \binom{a}{b}=0$ for $b<0$ and $b>a$.
Then, for $\ell=1, \dots, p^2-1$, 
the number of monomials of the form $z_1z_2z_{i_1}\cdots z_{i_{\ell-2}}$
 ($3 \leq i_1<\cdots <i_{\ell-2}\leq p^2$) is $\displaystyle \binom{p^2-2}{\ell-2}$,
both the number of monomials of the form
$z_1z_{i_1}\cdots z_{i_{\ell-1}}$ ($3 \leq i_1<\cdots <i_{\ell-1}\leq p^2$)
and the number of monomials of the form
$z_2z_{i_1}\cdots z_{i_{\ell-1}}$ ($3 \leq i_1<\cdots <i_{\ell-1}\leq p^2$)
are  $\displaystyle \binom{p^2-2}{\ell-1}$
and
the number of monomials of the form $z_{i_1}z_{i_2} \cdots z_{i_{\ell}}$ ($3 \leq i_1<\cdots < i_{\ell}\leq p^2$) is $\displaystyle \binom{p^2-2}{\ell}$.
Hence, we have 
\begin{align*}
\varphi_1^*(\lambda_{\ell}) & = \binom{p^2-2}{\ell-2} + \binom{p^2-2}{\ell-1} (z+z^{-1}) +\binom{p^2-2}{\ell}
\end{align*}
for $\ell=1, 2,\dots, p^2-1$.
 Therefore, the total Chern classes are given by  
 \begin{align*}
 c(\varphi_1^*(\lambda_\ell))&=\{(1-t)(1+t)\}^{ \binom{p^2-2}{\ell-1}}
 \\
 &=1-\binom{p^2-2}{\ell-1} t^2+\mbox{higher terms}, 
 \end{align*}
 where $\ell=1, \dots, p^2-1$ and higher terms 
 mean terms of topological degree greater than or equal to $6$.
 For $\ell=pk$, $0<k<p$, 
 by Lucas' theorem, since $p^2-2=(p-1)p+(p-2)$, $pk-1=(k-1)p+(p-1)$, we have
 \[
 \binom{p^2-2}{pk-1}\equiv \binom{p-1}{k-1} \binom{p-2}{p-1}\equiv 0 \mod p.
 \]
Therefore, we have $c_{2}(\varphi_1^{*}(\lambda_{pk}))\equiv 0 \mod p$.
 \end{proof}
 
 To compute second Chern classes of $\lambda_L$ for $r\geq 2$, we use the following lemma.
 
\begin{lemma}
\label{lemma:3.2}
Let $\xi, \eta$ be finite dimensional complex vector bundles over a CW complex of finite type. We assume that $H^{*}(X;\mathbb{Z})$ has no torsion and that the first Chern classes vanish, 
that is, $c_1(\xi)=c_1(\eta)=0$.
Then, we have 
\[
\dim \xi\eta = (\dim \xi)( \dim \eta), \;\;
c_1(\xi\eta)=0, \;\;
c_2(\xi\eta)=(\dim \eta) c_2(\xi)+(\dim \xi )c_2(\eta).
\]
\end{lemma}

 \begin{proof}
 We consider the Chern character $\mathrm{ch}(\xi)$
 of a finite dimensional complex vector bundle $\xi$ over $X$.
  By the assumption of the lemma, Chern characters $\mathrm{ch}(\xi)$, $\mathrm{ch}(\eta)$
  are
given by  \begin{align*}\mathrm{ch}(\xi)&=\dim \xi -c_2(\xi)+\mbox{higher terms,}
 \end{align*}
and
 \begin{align*}
 \mathrm{ch}(\eta)&=\dim \eta -c_2(\eta)+\mbox{higher terms.}
 \end{align*}
 where higher terms mean terms of topological degree greater than or equal to $6$
 as in the proof of Proposition~\ref{proposition:3.1}.
 The Chern character satisfies 
 \[
 \mathrm{ch}(\xi \eta)=\mathrm{ch}(\xi)\mathrm{ch}(\eta)
 \]
 for any finite dimensional complex representations $\xi$, $\eta$ over $X$.
 Hence, we have
 \begin{align*}
 \mathrm{ch}(\xi)\mathrm{ch}(\eta)
 &=(\dim \xi)( \dim \eta) -\{(\dim \eta) c_2(\xi)+(\dim \xi) c_2(\eta)\}+\mbox{higher terms,}
  \end{align*}
and we have the desired result.
 \end{proof}
 
 \begin{proof}[Proof of Theorem~\ref{theorem:1.1}]
 As we mentioned in  Section~\ref{section:1}, 
 Antieau proved that the induced homomorphism 
\[
\pi^*:H^4(BSU(p^2)/\mu_p;\mathbb{Z})\to H^4(BSU(p^2);\mathbb{Z})
\]
 is an isomorphism. Since we proved
\[
c_2(\varphi_1^*(\lambda_1))=(-1)\cdot t^2
\]
in the proof of Proposition~\ref{proposition:3.1},
it is clear that the induced homomorphism 
\[
\varphi_1^*:H^4(BSU(p^2);\mathbb{Z}) \to H^4(BS^1;\mathbb{Z})
\]
is also an isomorphism.
Hence, to prove Theorem~\ref{theorem:1.1}, 
it suffices to show the following:
\begin{itemize}
\item[(1)] 
$c_2(\varphi_1^*(x))$ is divisible by $p$ for all 
$x\in \mathrm{Im}\,\pi^* \subset {R}(SU(p^2))$ and 
\item[(2)] 
there exists an element $y\in R(SU(p^2)/\mu_p)$ such that 
$c_2(\varphi_1^*(\pi^*(y)))=p \cdot t^2$.
\end{itemize}

First, we prove (1).
Let  $L=(\ell_1, \dots, \ell_r)$
be a non-decreasing sequence of positive integers less than $p^2$.
For the sequence $L=(\ell_1, \dots, \ell_r)$,  let
\[
\delta_{L}= \prod_{i=1}^{r} \dim \varphi_1^*(\lambda_{\ell_i}) \quad \mbox{and} \quad \delta_{L,i}= 
\dfrac{\delta_{L}}{\dim \varphi_1^{*}(\lambda_i)}.
\]
By definition, we have
\[
\dim \varphi_1^*(\lambda_\ell)=\dim \lambda_\ell=\binom{p^2}{\ell}.
\]
It is divisible by $p$. 
If $r\geq 2$, 
$
\delta_{L,i}
$
 is the product of 
 $\dim \varphi_1^{*}(\lambda_{\ell_1}), \dots,  \dim \varphi_1^{*}(\lambda_{\ell_r})$ 
 except for $\dim \varphi_1^{*}(\lambda_{\ell_i})$. 
 Therefore, it is divisible by $p$ for $i=1, \dots, r$, that is, 
\[
\delta_{L,i} \equiv 0 \mod p
\]
 for $r\geq 2$.
On the other hand, 
by Proposition~\ref{proposition:3.1},  
$\varphi_1^*(\lambda_\ell)$'s  satisfy the assumption of Lemma~\ref{lemma:3.2}.
Therefore, by Lemma~\ref{lemma:3.2}, we have 
\[
c_2(\varphi_1^{*}(\lambda_L))=
\sum_{i=1}^r \delta_{L,i}\;
c_2(\varphi_1^{*}(\lambda_{\ell_i})).
\]
Hence, for $r\geq 2$, $c_2(\varphi_1^{*}(\lambda_L))\equiv 0 \mod p$.
Thus, modulo $p$, 
$c_2(\varphi_1^{*}(x))$ is a linear combination of $c_2(\varphi_1^{*}(\lambda_{\ell}))$, where
$\ell=1,\dots, p^2-1$.
Since $x\in \mathrm{Im}\,\pi^*$,
by Proposition~\ref{proposition:2.2},  for $\ell$  not divisible by $p$, we have
$
\alpha_{(\ell)}(x) =0.
$
Therefore,  we have
\begin{align*}
c_2(\varphi_1^{*}(x))
&
\equiv 
\sum_{k=1}^{p-1} \alpha_{(pk)}(x) c_2(\varphi_1^{*}(\lambda_{pk}))
\equiv
0\mod p,
 \end{align*}
by Proposition~\ref{proposition:3.1}.

Next, we prove (2).
Recall Remark~\ref{remark:2.1}. 
There exist complex representations $y_1, y_2$ such that $\pi^{*}(y_1)=\lambda_p$, $\pi^*(y_2)=\lambda_1^p$.
Let $L=(1, \dots, 1)$ be the sequence of $p$ copies of $1$'s so that
$\ell_1=\cdots=\ell_p=1$.
Since $\dim \lambda_{\ell_i}=p^2$ for $i=1, \dots, p$, we have
$
\delta_{L}=p^{2p}
$
 and 
 \[
\delta_{L,i}=\dfrac{ \delta_L}{\dim \varphi_1^{*}(\lambda_{\ell_i})}=p^{2p-2}
 \]
 for $i=1,\dots, p$.
Recall that we proved \[
c_2(\varphi_1^{*}(\lambda_1))=(-1) \cdot t^2
\]
in the proof of Proposition~\ref{proposition:3.1}.
By Lemma~\ref{lemma:3.2}, we have
\[
c_2(\varphi_1^{*}(\lambda_1^p))=\sum_{i=1}^p -p^{2p-2} \cdot t^2=-p^{2p-1} \cdot t^2.
\]
On the other hand,  in the proof of Proposition~\ref{proposition:3.1}, we proved
\[
c_2( \varphi_1^{*}(\lambda_p))=-\binom{p^2-2}{p-1}\cdot  t^2.
\]
Since
\[
\displaystyle \binom{p^2-2}{p-1}=\dfrac{p^2-p}{p-1} \left( \prod_{i=2}^{p-1} \dfrac{p^2-i}{p-i}\right) \equiv p \mod p^2,
\]
the greatest common divisor 
of $\displaystyle \binom{p^2-2}{p-1}$ and $p^{2p-1}$ is $p$,  so that there exist
integers $\beta_1, \beta_2$ such that 
\[
\beta_1 \binom{p^2-2}{p-1}
+ \beta_2 p^{2p-1}=-p.
\]
Let $y=\beta_1y_1+\beta_2y_2$. Then, $y$ satisfies the required conditions
and (2) holds.
\end{proof}

 \end{document}